\documentclass{article}

\usepackage{latexsym,epsfig,amssymb,amsmath,amsfonts,amsthm,url,bbm,enumerate,float, fancyhdr, hyperref}

\newtheorem{theorem}{Theorem}[section]
\newtheorem{lemma}[theorem]{Lemma}
\newtheorem{proposition}[theorem]{Proposition}
\newtheorem{corollary}[theorem]{Corollary}

\newtheorem{definition}[theorem]{Definition}
\newtheorem{Property}{Property}[section]

\newtheorem{remark}[theorem]{Remark}

\theoremstyle{definition}
\theoremstyle{definition} 
\newtheorem{Example}{Example}[section]

\def\P{\mathbb{P}}
\def\R{\mathbb{R}}

\newcommand\M{{\cal M}}

\begin{document}

\begin{center}
{\LARGE New results on the order of functions at infinity}

\bigskip

Meitner Cadena$^{1}$, Marie Kratz$^{2}$ and Edward Omey$^{3}$

\bigskip
\end{center}

$^{1}$ Universidad de las Fuerzas Armadas, Dept. de Ciencias Exactas,
Sangolqui, Ecuador [\textit{mncadena2@espe.edu.ec}]

$^{2}$ ESSEC Business School, CREAR\ risk research center [\textit{kratz@essec.edu}]

$^{3}$ KU\ Leuven @ Campus Brussels [\textit{edward.omey@kuleuven.be}]

\begin{center}
\textbf{Abstract} 
\end{center}
Recently, new classes of positive and measurable functions, $\M(\rho)$ and $\M(\pm \infty)$, have been defined  in terms of their asymptotic behaviour at infinity, when normalized by a logarithm (Cadena et al., 2015, 2016, 2017).  Looking for other suitable normalizing functions than logarithm seems quite natural. It is what is developed in this paper, 
studying new classes of functions of the type $\displaystyle \lim_{x\rightarrow \infty}\log U(x)/H(x)=\rho <\infty$ for a large class of normalizing functions $H$. It provides subclasses of $\M(0)$ and $\M(\pm\infty)$.
%
%

\section{Introduction}

Recently Cadena et al. (see \cite{CadenaRevisit-2015,CadenaEstim-2015,CadenaEstim-2015,CadenaT-2016,CK2015,CK2016,CKO2017}) introduced and studied the class of positive and measurable functions with support 
$\mathbb{R}^{+}$, bounded on finite intervals, such that
\begin{equation}\label{1}
\lim_{x\rightarrow \infty }\frac{\log U(x)}{\log x}=\rho 
\end{equation}
where $\rho $ is a finite real number, called {\it the order of the function $U$} given the use of such functions in complex analysis and entire functions. Notation: $U\in \mathcal{M}(\rho )$.
Relations of the form \eqref{1} extend the class of regularly varying functions. 
Recall that a positive and measurable function $U$ is regularly varying at infinity and with real index $\rho $, denoted by $U\in RV_{\rho }$, if it satisfies
\begin{equation}\label{2}
\lim_{x\rightarrow \infty }\frac{U(tx)}{U(x)}=t^{\rho }\text{, }\forall t>0 .
\end{equation}
It has already been proved in \cite{CKO2017} that $U\in RV_{\rho } \Rightarrow U\in \mathcal{M}(\rho )$ but that the converse is false. We also considered there $O-$ type of relations and, among
others, studied functions for which $\lim \sup_{x\rightarrow \infty }\log
U(x)/\log x<\infty$.

It is clear that \eqref{1} and \eqref{2} also make sense when $\rho =\infty $ or $\rho
=-\infty$. In general, it is not simple to characterize the corresponding
classes of functions $\mathcal{M}(\pm \infty )$, as seen in \cite{CKO2017}. It is also quite natural to look for other suitable normalizing function in \eqref{1}, than $\log x$. 
It turns out that there are different types of interesting relations. The first extension leads to a subclass of $\mathcal{M}(0)$, while the second one produces a family of subclasses of $\mathcal{M}(\pm \infty )$.
More precisely, in this paper we study functions $U$ characterized by the relation
\begin{equation}\label{eq:limH}
\lim_{x\rightarrow \infty}\frac{\log U(x)}{H(x)}=\rho <\infty 
\end{equation} 
for a large class of functions $H$.

This type of behavior can be encountered in various examples, hence our motivation for this general study. For instance, Bingham et al. considered \eqref{eq:limH} for $H\in RV_{\alpha }$ (see  (7.3.2) in \cite{BinghamGT1989}). 
Such property \eqref{eq:limH} appears also in connection with entire (complex) functions of the form $f(z)=\sum_{n=0}^{\infty }a_{n}z^{n}$ (see \cite{CKO2017}). 
Another example is the case of classical proximate order functions $\rho (.)$ defined by $\rho (x)=\log g(x)/\log x$, where $g$ is a regularly varying function. The class of functions $U$ satisfying  \eqref{eq:limH} lead then to {\it generalized proximate order functions} of the form $\rho (x)=\log g(x)/H(x)$, where $H$
may be different from the logarithmic function.
Another case, considered by several authors, concerns semi-exponential tail distributions of the form $\displaystyle \overline{F}(x)=\P[X>x]=A(x)\exp\{ -B(x)x^{\alpha }\}$ where $A,B\in RV_{0}$ and $0<\alpha \leq 1$. Clearly we have $\displaystyle -\log \overline{F}(x) = -\log A(x)+B(x)x^{\alpha }\underset{x\to\infty}{\thicksim} B(x)x^{\alpha }$ and \eqref{eq:limH}  holds. Gantert (see \cite{gantert98}) used those functions to obtain asymptotic expressions for $ \frac{\log \P[S_{n}/n\in A]}{n^{\alpha }B(n)}$, as $n\rightarrow \infty$, where $S_{n}$ denotes the $n$th partial sum. 
Let us mention a last recent example. 
In \cite{mimica}, Mimica considers the Laplace-Stieltjes transform $\displaystyle f(z)=\int_{[0,\infty )}\!\!\!\!\!e^{-zx}dF(x)$ where $F(.)$ is a measure on $[0,\infty)$. One of his main results provides conditions on $f(.)$ to make sure that 
\[
\lim_{x\rightarrow \infty }\frac{\log F(x,\infty )}{H(x)}=\alpha \text{,} 
\]
where $\alpha $ is a constant and $H$ is the identity function $H(x)=x$.\\[1ex]

\section{The classes $\mathcal{M}_{0}(L,\protect\rho )$ and $\mathcal{M}_{0}^{\pm }(L,\protect\rho )$}

\subsection{Definition}

In our effort to generalize \eqref{1}, we introduce the following classes of
functions. Throughout the paper, $L(.)$ denotes a positive and measurable function.

\begin{definition}
Consider the class of positive and measurable functions $U$ with support $\R^{+}$, bounded on finite intervals, such that  
\begin{equation} \label{3}
\lim_{x\rightarrow \infty }\frac{\log U(x)}{\int_{a}^{x}L(t)t^{-1}dt}=\rho 
\end{equation}
where $\rho$ and $a$ are real numbers. \\
%
If $\int_{a}^{\infty }L(t)t^{-1}dt<\infty $,
we consider functions $U$ for which we have
\begin{equation} \label{4}
\lim_{x\rightarrow \infty }\frac{\log U(x)}{\int_{x}^{\infty }L(t)t^{-1}dt} =\rho ,
\end{equation}
\end{definition}
\begin{remark}~
\begin{enumerate}
\item If $\displaystyle L(x) \underset{x\to\infty}{\rightarrow} c$, with $0<c<\infty$, then \eqref{3} implies
that $U\in \mathcal{M}(c\rho )$.
\item If \eqref{3} holds with $\int_{a}^{\infty }L(t)t^{-1}dt<\infty $, then $\log
U(x)\underset{x\to\infty}{\rightarrow} c<\infty$, a case with not so much interest.
\end{enumerate}
\end{remark}

Let us consider various possible behaviors for the positive and measurable function $L$ that we present as  different sets of assumptions on $L$. 

\begin{itemize}
\item[]\textbf{Assumption A.} 
Assume that 
$$
L(x) \underset{x\to\infty}{\rightarrow} 0 \quad \text{and} \quad  \int_{a}^{\infty }L(t)t^{-1}dt=\infty \quad
(a\in\R).
$$
\end{itemize}
Under Assumption~A, the class of functions satisfying \eqref{3} will be denoted by $\mathcal{M}_{0}(L,\rho )$.
\begin{itemize}
\item[]\textbf{Assumption B.}  Assume that \;
$\displaystyle 
L(x)\, {\rightarrow} \,\infty  \quad \text{as}\; x\to \infty
$.\\
Note that this assumption implies that $\displaystyle \int_{a}^{\infty }L(t)t^{-1}dt=\infty $.
\end{itemize}
Under Assumption~B, the class of functions satisfying \eqref{3} will be denoted by $\mathcal{M}_{0}^{+}(L,\rho )$.
\begin{itemize}
\item[]\textbf{Assumption C.} Assume that \;
$\displaystyle \int_{a}^{\infty }L(t)t^{-1}dt<\infty$,\; with $a\in\R$.
\end{itemize}
Under Assumption~C, the class of functions satisfying \eqref{4} will be denoted by $\mathcal{M}_{0}^{-}(L,\rho )$. 
This class of functions will be discussed only briefly in this paper.

\subsection{Examples}

Standard functions can be encountered in those classes, as we can see in the following examples.\\[-4ex]
\begin{enumerate}
\item If $U(x)=(\log x)^{\rho }$, then $U\in \mathcal{M}_{0}(L,\rho )$ with $L(x)=1/\log x$.
\item Let $U(x)=e^{\rho\, (\log x)^{\beta}}$.\\
If  $0<\beta <1$, then $U\in \mathcal{M}(0)\cap \mathcal{M}_{0}(L,\rho )$ with $L(x)=\beta (\log
x)^{\beta -1}$. \\
If $\beta <0$, then $U\in \mathcal{M}(0)\cap \mathcal{M}_{0}^{-}(L,\rho )$ with $L(x)=-\beta (\log x)^{\beta -1}$.
\item If $U(x)=e^{\rho\, x^{-\alpha}}$, $\alpha >1$, then $U\in \mathcal{M}(0)\cap \mathcal{M}_{0}^{-}(L,\rho )$ with $L(x)=\alpha x^{-\alpha }$.

\item If $U(x)=e^{\rho\, x^{\alpha}}$, $\alpha >0$, then $U\in \mathcal{M}_{0}^{+}(L,\rho )$ with $L(x)=\alpha x^{\alpha }$.

\item Assume that a function $U$ has a derivative $U^{\prime }$ such that  
$\displaystyle \frac{x\,U^{\prime}(x)}{L(x)U(x)} \underset{x\to\infty}{\rightarrow} \alpha$, with $\alpha\in \R$. For each $\varepsilon >0$, we can find $x_\varepsilon$ such that, for $x\geq x_\varepsilon$,
\[
(\alpha -\varepsilon )\frac{L(x)}{x}\leq \frac{U^{\prime }(x)}{U(x)}\leq
(\alpha +\varepsilon )\frac{L(x)}{x}.
\]
Under Assumption~A,  $U\in \mathcal{M}_{0}(L,\alpha)$; under Assumption~B,  $U\in \mathcal{M}_{0}^{+}(L,\alpha )$; under Assumption~C, $U\in \mathcal{M}_{0}^{-}(L,\alpha)$.
%
\item Suppose that $a\in RV_{0}$ and that $U\in RV_{0}$ is a positive and
measurable function. We say that $U$ belongs to the class $\Pi (a)$ if $U$
satisfies, $\forall t>0$,
\begin{equation} \label{5}
\lim_{x\rightarrow \infty }\frac{U(tx)-U(x)}{a(x)}=\log t .
\end{equation}
The class $\Pi (a)$ was introduced and studied among others by de Haan (see \cite{deHaan}). 
If \eqref{5} holds, we have (see \cite{BinghamGT1989}, Chapter 3) 
\[
\log U(tx)-\log U(x)=\log \frac{U(tx)}{U(x)} \underset{x\to\infty}{\thicksim} \frac{U(tx)-U(x)}{U(x)}
\underset{x\to\infty}{\thicksim} \frac{a(x)\log t}{U(x)}\text{,} 
\]
so that $\log U\in \Pi (L)$ with $L(x):=a(x)/U(x)\in RV_{0}$. From the representation theorem for the class $\Pi $ (see Theorem 3.7.3(ii) in \cite{BinghamGT1989}), it follows that $\log U$ can be written as
\begin{equation}\label{6}
\log U(x)=C+A(x)+\int_{a}^{x}A(t)t^{-1}dt,\quad x\geq a , 
\end{equation}
where $A\in RV_{0}$ and $A(x)\underset{x\to\infty}{\thicksim} L(x)$. Now consider two cases.\\[-4ex]
\begin{itemize}
\item[(i)] If $\displaystyle \int_{a}^{x}A(t)t^{-1}dt \underset{x\to\infty}{\rightarrow} \infty $, then Karamata's theorem shows that \\$\displaystyle A(x)=o(1)\int_{a}^{x}A(t)t^{-1}dt$ and \eqref{6} implies that $\displaystyle U\in \mathcal{M}_{0}(L,1)$.
\item[(ii)]  If $\int_{a}^{\infty }A(t)t^{-1}dt<\infty $, then \eqref{6} implies that\\
$\displaystyle \log U(x)\underset{x\to\infty}{\rightarrow}  C+\int_{a}^{\infty }A(t)t^{-1}dt=:D$; so, for $x\ge a$,  we have\\
$\displaystyle D-\log U(x) \underset{x\to\infty}{\thicksim} \int_{x}^{\infty }A(t)t^{-1}dt $.
We can then deduce, since \\ $A(x)\underset{x\to\infty}{\thicksim} L(x)$, that $\displaystyle V:=e^{D}/U$ belongs to the class $\displaystyle \mathcal{M}_{0}^{-}(L,1)$.
\end{itemize}
\end{enumerate}

\subsection{The class $\mathcal{M}_{0}(L,\protect\rho )$}

\subsubsection{First properties}

\begin{proposition}
Under Assumption~A, if $U\in \mathcal{M}_{0}(L,\rho )$ then $U\in \mathcal{M}(0)$.
\end{proposition}

\begin{proof}
We have, using  L'Hopital's rule,
$$
\lim_{x\rightarrow \infty }\frac{\log U(x)}{\log x} =\lim_{x\rightarrow
\infty }\frac{\log U(x)}{\int_{a}^{x}L(t)t^{-1}dt}\times \frac{\int_{a}^{x}L(t)t^{-1}dt}{\log x} 
=\rho \times \lim_{x\rightarrow \infty }\frac{\int_{a}^{x}L(t)t^{-1}dt}{\log x}=0.
$$
\end{proof}

The following result shows that the class $\mathcal{M}_{0}(L,\rho )$ does {\it not}
contain regularly varying functions with index different from $0$.

\begin{proposition}
Suppose that Assumption~A holds and that $U\in RV_{\alpha }$ with $\alpha
\neq 0$. Then $U\notin \mathcal{M}_{0}(L,\rho )$ for any $\rho$.
\end{proposition}

\begin{proof}
The representation theorem for regularly varying functions (see \cite{BinghamGT1989}, Theorem 1.3.1) states that $U$ can be written as $\displaystyle U(x)=c(x)\exp\left\{ \int_{a}^{x}\beta (t)t^{-1}dt\right\}$,  
where $c(x)\underset{x\to\infty}{\rightarrow} c\neq 0$ and $\beta (x)\underset{x\to\infty}{\rightarrow}  \alpha $. Taking logarithms in this expression provides, under Assumption~A,  
$$
\lim_{x\to\infty}  \frac{\log U(x)}{\int_{a}^{x}L(t)t^{-1}dt}=\lim_{x\to\infty}  \frac{\int_{a}^{x}\beta (t)t^{-1}dt}{\int_{a}^{x}L(t)t^{-1}dt}=\lim_{x\to\infty}  \frac{\beta (x)}{L(x)}.
$$
This limit is not finite since $\alpha \neq 0$ and $L(x)\underset{x\to\infty}{\rightarrow}  0$.
Hence the result.
\end{proof}

In our next result, we collect some algebraic results.
\begin{lemma}\label{lem:proptiesA}
Suppose $U\in \mathcal{M}_{0}(L_{1},\alpha
) $, with $L_1$ satisfying Assumption A.
\begin{itemize}
\item[(i)] If $V\in \mathcal{M}_{0}(L_{1},\beta )$, then $UV\in \mathcal{M}_{0}(L_{1},\alpha +\beta )$ and $U/V\in \mathcal{M}_{0}(L_{1},\alpha -\beta) $.
\item[(ii)] If $V\in \mathcal{M}_{0}(L_{2},\beta )$ and $L_{2}(x)/L_{1}(x)\underset{x\to\infty}{\rightarrow}  0$, then $UV\in \mathcal{M}_{0}(L_{1},\alpha )$.
\item[(iii)] If \, $xV^{\prime }(x)/V(x)\underset{x\to\infty}{\rightarrow}  \beta >0$, then \, $U\circ V\in 
\mathcal{M}_{0}(L\circ V,\alpha \beta )$.
\end{itemize}
\end{lemma}

\begin{proof}
Properties (i) and (ii)  follow directly from the definition \eqref{3}.\\
Let us check (iii). The condition $xV^{\prime }(x)/V(x)\rightarrow \beta >0$ ensures that, for large values of $x$, $V^{\prime }(x)>0$, {\it i.e.} $V(x)$ is increasing, implying $V(x) \underset{x\to\infty}{\rightarrow}\infty $. Hence, from \eqref{3}, it follows that
$\displaystyle \frac{\log U(V(x))}{\int_{a}^{V(x)}L(t)t^{-1}dt} \, \underset{x\to\infty}{\rightarrow} \alpha$. \\
But the change of variables $t=V(z)$ gives
$\displaystyle \int_{a}^{V(x)}L(t)t^{-1}dt=\int_{b}^{x}L(V(z))\frac{zV^{\prime }(z)}{V(z)} z^{-1}dz$.
Then, using l'Hopital's rule, we obtain that 
$\displaystyle \frac{\int_{a}^{V(x)}L(t)t^{-1}dt} {\int_{b}^{x}L(V(z))z^{-1}dz}\underset{x\to\infty}{\rightarrow} \beta$. \\Combining the two limits provides the result.
\end{proof}


\subsubsection{Characterization theorem}

We can obtain a characterization theorem for functions belonging to $\mathcal{M}_{0}(L,\rho )$.
We denote in what follows, for $a\in \R$ and $x>a$,
\begin{equation}\label{V}
\displaystyle V(x):=\exp \left\{\int_{a}^{x}L(t)t^{-1}dt\right\}.
\end{equation} 

%
\begin{theorem}\label{th:charact} 
Under Assumption~A, we have:
$$
U\in \mathcal{M}_{0}(L,\rho ) \quad \Longleftrightarrow\quad
\forall \epsilon >0,  \; \frac{U(x)}{V^{\rho+\epsilon }(x)}\underset{x\to\infty}{\rightarrow} 0 \quad\text{and}\quad \frac{U(x)}{V^{\rho -\epsilon }(x)}\underset{x\to\infty}{\rightarrow}\infty.
$$
\end{theorem}
It is straightforward to deduce the following.
\begin{corollary}\label{cor:7}
Under the conditions of Theorem \ref{th:charact}, if $U\in \mathcal{M}_{0}(L,\rho )$, then for
each $\epsilon >0$, there exists $x_\epsilon>0$ such that, for all $x\ge x_\epsilon$, 
\begin{equation}\label{7}
V^{\rho -\epsilon }(x)\leq U(x)\leq V^{\rho +\epsilon }(x).
\end{equation}
\end{corollary}

{\it Proof of Theorem \ref{th:charact}.}
The proof of the theorem is based on the following lemma, inspired by Theorem 1.1 in \cite{CadenaT-2016} or \cite{CKO2017}.
\begin{lemma}\label{lem:meitner}
Let $U$ and $g$ denote measurable and positive functions with support $\R^{+}$ and assume that $g(x)\underset{x\to\infty}{\rightarrow} \infty $. Let $\rho \in \R $. The following statements are equivalent: 
\begin{itemize}
\item[(i)] $\log U(x)/g(x)\underset{x\to\infty}{\rightarrow}  \rho $;
\item[(ii)] $\forall \epsilon >0$, \,$U(x)\,e^{-(\rho +\epsilon )g(x)} \underset{x\to\infty}{\rightarrow}  0$
\, and \, $U(x)\,e^{ -(\rho -\epsilon )g(x)}\underset{x\to\infty}{\rightarrow}  \infty $.
\end{itemize}
\end{lemma}

{\it Proof of Lemma \ref{lem:meitner}.}
First suppose that (i) holds and let $\epsilon>0$. For each $\delta \in (0,\epsilon )$, there exists $x_\delta$ such that for all $x\geq x_\delta$,
\[
(\rho -\delta )g(x)\leq \log U(x)\leq (\rho +\delta )g(x).
\]
It follows that $\displaystyle
\frac{U(x)}{\exp\{ (\rho +\epsilon )g(x)\}}\leq \exp\{(\delta -\epsilon)g(x)\}\, \underset{x\to\infty}{\rightarrow} \, 0
$
and, similarly,
$\displaystyle
\frac{U(x)}{\exp\{(\rho -\epsilon )g(x)\}}\geq \exp\{(\epsilon -\delta)g(x)\}\,\underset{x\to\infty}{\rightarrow} \, \infty$,
hence the result (ii).

Now assume that (ii) holds. For each $\delta >0$, there exists $x_\delta$ such that for all $x\ge x_\delta$, $\displaystyle U(x)\,e^{- (\rho +\epsilon )g(x)}\leq \delta$.
It follows that $\log U(x)\leq \log \delta +(\rho +\epsilon )g(x)$ and then
also that $\displaystyle \limsup_{x\to\infty} \log U(x)/g(x)\leq \rho +\epsilon $. In a similar way, we
obtain that $\displaystyle \liminf_{x\to\infty}  \log U(x)/g(x)\geq \rho -\epsilon $. Now (i) follows. $\Box$

Note that Assumption~A implies that $V\in RV_{0}$, $\displaystyle \lim_{x\rightarrow \infty }V(x)=\infty $ and
that $xV^{\prime }(x)/V(x)=L(x)\underset{x\to\infty}{\rightarrow} 0$. Applying Lemma \ref{lem:meitner} with $g:=V$ concludes the proof of the theorem. 
$\Box$

\subsubsection{Representation theorem}

We can also provide a representation theorem for the class $\mathcal{M}_{0}(L,\rho )$, as follows.
\begin{theorem}
Suppose Assumption A holds. Then $\displaystyle U\in \mathcal{M}_{0}(L,\rho )$ if and only if $U$ is of the form
\begin{equation}\label{8}
\log U(x)=\alpha (x)+\int_{a}^{x}\beta (t)L(t)t^{-1}dt
\end{equation}
where $\displaystyle \alpha (x)=o(1)\int_{a}^{x}L(t)t^{-1}dt$ and $\beta (x)\rightarrow\rho$, as $x\to\infty$.
\end{theorem}

\begin{proof}
From definition \eqref{3}, it follows that
\[
\log U(x)=(\rho +\epsilon (x))\int_{a}^{x}L(t)t^{-1}dt\,,\quad\text{ with} \; \epsilon (x)\underset{x\to\infty}{\rightarrow} 0.
\]
Therefore we can write
\[
\log U(x)=\alpha (x)+\int_{a}^{x}\beta (t)L(t)t^{-1}dt 
\]
where $\beta (x)=\rho +\epsilon (x)$\, and \;
$\displaystyle  \alpha (x)=\epsilon (x)\int_{a}^{x}L(t)t^{-1}dt-\int_{a}^{x}\epsilon
(t)L(t)t^{-1}dt$.
Clearly $\displaystyle \alpha (x)=o(1)\int_{a}^{x}L(t)t^{-1}dt$ and $\beta(x) \rightarrow \rho$, as $x\to\infty$, hence \eqref{8}. \\
The converse result is obvious.
\end{proof}
 
\subsubsection{Integrals}

\begin{proposition} (Karamata's theorem). Let $U\in \mathcal{M}_{0}(L,\rho )$. Under Assumption A, we have:
\begin{itemize}
\item[(i)] If $\alpha >-1$, then \; $\displaystyle x^{-1-\alpha}\!\!\int_{a}^{x}t^{\alpha }U(t)\, dt \; \in 
\mathcal{M}_{0}(L,\rho )$.
\item[(ii)]  If $\alpha <-1$, then \; $\displaystyle x^{-1-\alpha}\!\!\int_{x}^{\infty }t^{\alpha
}U(t)\, dt \; \in \mathcal{M}_{0}(L,\rho )$.
\end{itemize}
\end{proposition}

\begin{remark} 
Proving the converse result  is an open problem. Moreover, it is not clear what happens when $\alpha =-1$.
\end{remark}

\begin{proof}~\\
(i) First consider the case where $\alpha >-1$. We choose $\epsilon>0$ and use \eqref{7} . Multiplying by $x^{\alpha }$ and taking integrals, we find that
\[
\int_{a}^{x_\epsilon} t^{\alpha }U(t)dt+\int_{x_\epsilon}^{x}t^{\alpha }V^{\rho -\epsilon }(t)dt\leq \int_{a}^{x}t^{\alpha
}U(t)dt\leq \int_{a}^{x_\epsilon} t^{\alpha }U(t)dt+\int_{x_\epsilon}^{x}t^{\alpha }V^{\rho +\epsilon }(t)dt\text{.} 
\]
Since $V\in RV_{0}$, we find that $\displaystyle V^{\rho \pm \epsilon}\in RV_{0}$, so $t^\alpha V^{\rho \pm \epsilon }(t)\in RV_{\alpha}$. Applying Karamata's theorem (see \cite{BinghamGT1989}, Proposition 1.5.8) gives that, under Assumption A, 
\[
\int_{x_\epsilon}^{x}t^{\alpha }V^{\rho \pm \epsilon }(t)dt\; \underset{x\to\infty}{\thicksim}\; \frac{1}{1+\alpha } x^{1+\alpha }V^{\rho \pm \epsilon }(x) \, \underset{x\to\infty}{\rightarrow}\, \infty .
\]
It follows that, there exists $x_0>x_\epsilon$ such that, for all $x\ge  x_0$, 
\[
\frac{1-\epsilon }{1+\alpha }V^{\rho -\epsilon }(x)\,\leq\, \frac{1}{x^{1+\alpha
}}\int_{a}^{x}t^{\alpha }U(t)\,dt \,\leq \,\frac{1+\epsilon }{1+\alpha }V^{\rho
+\epsilon }(x).
\]
Using \, $\log V(x) \underset{x\to\infty}{\rightarrow} \infty $ and taking $\epsilon\to 0$, we conclude that 
\[
\frac{1}{x^{1+\alpha }}\int_{a}^{x}t^{\alpha }U(t)\,dt \,\in \mathcal{M}
_{0}(L,\rho )\text{.} 
\]

(ii) In the case where $\alpha <-1$, the proof follows, using similar steps as for (i) with now (via Karamata) \,
$\displaystyle \int_{x}^{\infty }t^{\alpha }V^{\rho \pm \epsilon }(t)dt\thicksim \frac{-1}{1+\alpha }x^{1+\alpha }V^{\rho \pm \epsilon }(x)$.
\end{proof}

\subsubsection{Laplace transforms}

Recall that the Laplace transform of $U$ is given by $\displaystyle \widehat{U}(s)=s\int_{0}^{\infty }e^{-sx}U(x)dx$.
Assume now the conditions given in \cite{CKO2017}, Lemma 1.2, namely that $U$ is a nondecreasing right continuous function with support $\R ^{+}$and $U(0+)=0$,  that $\widehat{U}(s)<\infty ,\forall s>0$, and that $x^{-\eta }U(x)$ is a concave function for some real number $\eta >0$.
%
%
Then we have the following Karamata Tauberian type of theorem.

\begin{theorem}
Under Assumption A and the conditions of Lemma~1.2 in \cite{CKO2017} (recalled above), assuming $\rho>0$, we have: 
$$
U\in \mathcal{M}_{0}(L,\rho )\quad\text{ if and only if} \quad\widehat{U}(1/x)\in \mathcal{M}_{0}(L,\rho ).
$$
\end{theorem}

\begin{proof}
First suppose that $U\in \mathcal{M}_{0}(L,\rho )$.  Using \eqref{7} with $V$ defined in \ref{V}, and Lemma~1.2 in \cite{CKO2017}, we can say that there exist positive constants $a,b,c$, such that, $\forall x>x_{\epsilon}$,
\[
a\,V^{\rho -\epsilon }(x)\, \leq \, \widehat{U}(1/x) \,\leq \, b\,V^{\rho +\epsilon}(cx).
\]
Since $V\in RV_{0}$, it follows that $\widehat{U}(1/x)\in \mathcal{M}_{0}(L,\rho )$. The converse implication can be proved in a similar way, using the same lemma.
\end{proof}

In fact, the implication $U\in \mathcal{M}_{0}(L,\rho )\Longrightarrow \widehat{U}(1/x)\in \mathcal{M}_{0}(L,\rho )$ can be proved without the concavity
condition:
\begin{proposition}
Let $U\in \mathcal{M}_{0}(L,\rho )$. Assume that $U(.)$ and $V(.)$ are bounded on bounded intervals and that $\widehat{U}(s)<\infty $, $\forall s>0$. Then, under Assumption A, $\widehat{U}(1/x)\in \mathcal{M}_{0}(L,\rho )$.
\end{proposition}

\begin{proof}
Using \eqref{7}, we have
$$
s\!\!\int_{0}^{x_\epsilon} \!\!\!\! e^{-sx}U(x)dx+s\!\!\int_{x_\epsilon}^{\infty } \!\!\! \!e^{-sx}V^{\rho -\epsilon }(x)dx\leq \widehat{U}(s)
\leq s\!\!\int_{0}^{x_\epsilon} \!\!\!\! e^{-sx}U(x)dx+s\!\!\int_{x_\epsilon}^{\infty } \!\!\!\! e^{-sx} V^{\rho +\epsilon }(x)dx\text{.} 
$$
Since $V\in RV_{0}$, we have $s\int_{0}^{\infty }e^{-sx}V^{\rho \pm \epsilon
}(x)dx \underset{s\to0}{\thicksim} V^{\rho \pm \epsilon }(1/s)$ (see \cite{BinghamGT1989}, Theorem 1.7.1). Also, since $V^{\rho \pm \epsilon }$
and $U$ are bounded on bounded intervals, we have
$$
s\int_{0}^{x_\epsilon}e^{-sx}V^{\rho \pm \epsilon }(x)dx =O(1) \quad\text{and}\quad
s\int_{0}^{x_\epsilon} e^{-sx}U(x)dx =O(1),
$$
hence $\displaystyle (1-\epsilon )V^{\rho -\epsilon }(s)+O(1)\leq \widehat{U}(1/s)\leq (1+\epsilon )V^{\rho +\epsilon }(s)+O(1)$. \\
Since, under Assumption A, $V(x)\rightarrow \infty $, we find that $\widehat{U}(1/x)\in \mathcal{M}_{0}(L,\rho )$.
\end{proof}

\textbf{Remark}. We could also consider $O-$versions of this class of functions.

\subsection{The class $\mathcal{M}_{0}^{-}(L,\protect\rho )$}
 
Recall that for the class $\mathcal{M}_{0}^{-}(L,\rho )$, we assume that 
$L$ satisfies Assumption~C. It implies that $L(x)\underset{x\rightarrow\infty}{\rightarrow} 0$. 
Moreover if $U\in \mathcal{M}_{0}^{-}(L,\rho )$, then $ \log U(x) \underset{x\to\infty}{\rightarrow} 0$, hence $U(x) \underset{x\to\infty}{\rightarrow} 1$. The following relation follows.
\begin{proposition}
Under Assumption C,  
$$
\eqref{4}\,\text{holds}\quad \Longleftrightarrow \quad \lim_{x\rightarrow \infty }\frac{U(x)-1}{\int_{x}^{\infty }L(t)t^{-1}dt}=\rho .
$$
\end{proposition}

{\it Remark}. The proposition implies that $\mathcal{M}_{0}^{-}(L,\rho )\subset RV_{0}$.

In the next result we obtain a representation theorem.

\begin{theorem}
Under Assumption~C, we have
$$
U\in \mathcal{M}_{0}^{-}(L,\rho )\quad\Leftrightarrow \quad 
\log U(x)=\alpha (x)+\int_{x}^{\infty }\beta (t)L(t)t^{-1}dt,
$$
where, as $x\to\infty$, $\beta (x)\rightarrow \rho $ and $\displaystyle \alpha (x)=o(1)\int_{x}^{\infty
}L(t)t^{-1}dt$.
\end{theorem}

\begin{proof}
First assume $U\in \mathcal{M}_{0}^{-}(L,\rho )$. Define the function $\varepsilon$, with $\varepsilon (x)\underset{x\to\infty}{\rightarrow} 0$, as
\[
\varepsilon (x)=\frac{\log U(x)}{\int_{x}^{\infty }L(t)t^{-1}dt}-\rho \text{.
} 
\]
We have $\displaystyle \log U(x) =(\varepsilon (x)+\rho )\int_{x}^{\infty }L(t)t^{-1}dt =\alpha (x)+\int_{x}^{\infty }\beta (t)L(t)t^{-1}dt$,
where $\beta (x)=\varepsilon (x)+\rho $ and 
$\displaystyle \alpha (x)=\varepsilon (x)\int_{x}^{\infty }L(t)t^{-1}dt-\int_{x}^{\infty}\varepsilon (t)L(t)t^{-1}dt$.
The result follows. The converse result is straighforward.
\end{proof}

\subsection{The class $\mathcal{M}_{0}^{+}(L,\protect\rho )$}

Here we consider $U\in \mathcal{M}_{0}^{+}(L,\rho )$, and {throughout this section we assume
that $L$ satisfies Assumption~B.

\subsubsection{Some properties}

\begin{proposition}~ 
\begin{itemize}
\item[(i)] Suppose that $U\in \mathcal{M}_{0}^{+}(L,\rho )$.
If $\rho >0$, then $U\in \mathcal{M}(\infty )$.\\ 
If $\rho <0$, then $U\in  \mathcal{M}(-\infty )$.
\item[(ii)] Suppose that $U\in \mathcal{M}(\rho )$. Then $U\in \mathcal{M}_{0}^{+}(L,0)$.
\end{itemize}
\end{proposition}

\begin{proof}
(i) We have 
$$
\lim_{x\rightarrow \infty }\frac{\log U(x)}{\log x} =\lim_{x\rightarrow\infty }\frac{\log U(x)}{\int_{a}^{x}L(t)t^{-1}dt}\times \frac{
\int_{a}^{x}L(t)t^{-1}dt}{\log x} =\rho \times \lim_{x\rightarrow \infty }\frac{\int_{a}^{x}L(t)t^{-1}dt}{\log x}.
$$
The result follows since $\displaystyle \lim_{x\rightarrow \infty}\int_{a}^{x}L(t)t^{-1}dt/\log x=\infty $ by l'Hopital's rule. 

(ii) Using the representation theorem for $\mathcal{M}(\rho )$ (see \cite{CKO2017}, Theorem 1.2), we have
\[
U(x)=c(x)\exp \int_{a}^{x}\beta (t)t^{-1}dt\text{,} \; \text{with}\;\log c(x)/\log x\underset{x\to\infty}{\rightarrow} 0\; \text{and}\;\beta (x)\underset{x\to\infty}{\rightarrow}  \rho.
\]
Taking logarithms, we obtain, using l'Hopital's rule and Assumption B, 
\[
\lim_{x\rightarrow \infty } \frac{\log U(x)}{\int_{a}^{x}L(t)t^{-1}dt}=\lim_{x\rightarrow \infty } \frac{\log c(x)}{\log x}\frac{\log x}{\int_{a}^{x}L(t)t^{-1}dt}+\lim_{x\rightarrow \infty } \frac{\int_{a}^{x}\beta
(t)t^{-1}dt}{\int_{a}^{x}L(t)t^{-1}dt}=0. 
\vspace{-3ex}
\]
\end{proof}

In our next result, we collect some algebraic results.
\begin{lemma}
Suppose that $U\in \mathcal{M}_{0}^{+}(L_{1},\alpha )$.

(i) If $V\in \mathcal{M}_{0}^{+}(L_{1},\beta )$, then $UV\in \mathcal{M}_{0}^{+}(L_{1},\alpha +\beta )$ and $U/V\in \mathcal{M}_{0}^{+}(L_{1},\alpha
-\beta )$;

(ii) If $V\in \mathcal{M}_{0}^{+}(L_{2},\beta )$ and $L_{2}(x)/L_{1}(x)
\underset{x\to\infty}{\rightarrow}  0$, then $UV\in \mathcal{M}_{0}^{+}(L_{1},\alpha )$;

(iii) If $xV^{\prime }(x)/V(x)\underset{x\to\infty}{\rightarrow}  \beta >0$, then $U\circ V\in 
\mathcal{M}_{0}^{+}(K,\alpha \beta )$, where $K=L\circ V$.
\end{lemma}

\begin{proof}
(i) and (ii) follow from the definition, whereas (iii) follows as in Lemma~\ref{lem:proptiesA}.
\end{proof}

\subsubsection{Characterization theorem}

Following the proof and the notation of \eqref{V}, Lemma~\ref{lem:meitner} and Theorem~\ref{th:charact},  we have the
following result.
\begin{theorem}\label{theo:16}
We have the following equivalence:
$$
 U\in \mathcal{M}_{0}^{+}(L,\rho ) \quad \Leftrightarrow \quad
\forall \epsilon >0,\; U(x)/V^{\rho+\epsilon }(x)\underset{x\to\infty}{\rightarrow}  0 \;\, \text{and}\;\,  U(x)/V^{\rho -\epsilon }(x)\underset{x\to\infty}{\rightarrow} \infty.
$$
\end{theorem}

\begin{remark}
Let $x>0$ and $t>1$.\ We have 
\[
\frac{V(xt)}{V(x)}=\exp\left\{ \int_{1}^{t}L(xy)y^{-1}dy \right\}.
\]%
Since $L(x)\underset{x\to\infty}{\rightarrow}  \infty $, it follows that, for $t>1$,  $V(xt)/V(x)\underset{x\to\infty}{\rightarrow}  \infty $, whereas for $t<1$, we have $V(xt)/V(x)\underset{x\to\infty}{\rightarrow}  0$. 
It shows that, under Assumption B,  $V$ is rapidly varying. In Section~\ref{sec:M1}, we will discuss other conditions on $V$.
\end{remark}

\subsubsection{Integrals}

In this section, we consider integrals of the form $\int_{0}^{x}U(t)dt$ or $\int_{x}^{\infty }U(t)dt$. 
\begin{proposition}
Introduce the function $W$ defined by $\displaystyle W(x)=x/L(x)$.
Assume that $U$ is differentiable and satisfies $\displaystyle U(x) \underset{x\to\infty}{\thicksim} \frac1{\rho}\, W(x)U^{\prime }(x)$,  $\rho >0$. Then we have:
\begin{itemize}
\item[(i)] If \, $\displaystyle W^{\prime }(x)>0$ and \, $W^{\prime }(x)\underset{x\to\infty}{\rightarrow}\alpha \geq 0$, then \; $\displaystyle U\in \mathcal{M}_{0}^{+}(L,\rho )$, 
$$
\int_{a}^{x}U(t)dt \underset{x\to\infty}{\thicksim} \frac1{\rho+\alpha}\,W(x)U(x) \quad \text{and}\quad \int_{a}^{x}U(t)dt\in \mathcal{M}_{0}^{+}(L,\rho +\alpha ).
$$ 
\item[(ii)] If \, $\displaystyle W^{\prime}(x)\underset{x\to\infty}{\rightarrow} \infty$, then \; $\displaystyle \int_{a}^{x}U(t)dt=o(1)W(x)U(x)$ \; and \\ $\displaystyle \int_{0}^{x}U(t)dt\in \mathcal{M}_{0}^{+}(L,0)$.
\end{itemize}
\end{proposition}

\begin{proof}
(i) Clearly $U(x)\underset{x\to\infty}{\thicksim} \frac1\rho \,W(x)U^{\prime }(x)\; \Leftrightarrow \; \frac{U'(x)}{U(x)}\,\underset{x\to\infty}{\thicksim} \,\rho \,x^{-1}L(x)$; it  implies (as already seen in Section 2.2, Example 5) that $\displaystyle U\in \M_{0}^{+}(L,\rho )$, 
$ U(x)\underset{x\to\infty}{\rightarrow} \infty $ 
and there exists $a$ such that \, $\displaystyle \int_{a}^{x}U(t)dt \underset{x\to\infty}{\thicksim} \frac{1}{\rho }\int_{a}^{x}W(t)U^{\prime }(t)dt $.\\
Now consider 
$\displaystyle  R(x)=\frac{W(x)U(x)}{\int_{a}^{x}W(t)U^{\prime }(t)dt}$.
Applying l'Hopital's rule, we obtain that 
$$
\lim_{x\rightarrow \infty }R(x) =\lim_{x\rightarrow \infty }\frac{W(x)U^{\prime }(x)+W^{\prime }(x)U(x)}{W(x)U^{\prime }(x)} 
= 1+\rho ^{-1}\lim_{x\rightarrow \infty }W^{\prime }(x)=1+\alpha / \rho.
$$
It follows that
$\displaystyle \int_{a}^{x}U(t)dt \underset{x\to\infty}{\thicksim} \frac{W(x)U(x)}{\rho (1+\alpha/\rho) }$, hence $\displaystyle\int_{a}^{x} \!\!\! U(t)dt\in \mathcal{M}_{0}^{+}(L,\rho +\alpha )$, when applying  Example 5 in Section 2.2 for the function $\displaystyle \int_{a}^{x}U(t)dt$.

(ii) If $W^{\prime }(x)\underset{x\to\infty}{\rightarrow} \infty$, then $R(x)\!\!\underset{x\to\infty}{\to}\!\!\infty$ and $\displaystyle \frac{\int_{a}^{x}U(t)dt }{W(x)U(x) }\underset{x\to\infty}{\thicksim} \frac1{\rho\,R(x)}\underset{x\to\infty}{\rightarrow} 0$.
\end{proof}

\begin{remark} ~\\[-5ex]
\begin{enumerate}
\item A similar result holds for the case where $\rho <0$ and the integral $\displaystyle\int_{x}^{\infty }U(t)dt$.
\item Assume \;$\displaystyle \frac{x\,W^{\prime }(x)}{W(x)}\underset{x\to\infty}{\rightarrow}\beta $ with $0<\beta <1$. Then $W\in RV_{\beta }$ and \\$\displaystyle W^{\prime }(x)\underset{x\to\infty}{\rightarrow} 0$.
\end{enumerate}
\end{remark}
A related result follows directly from Theorem 4.12.10 in \cite{BinghamGT1989}.
\begin{proposition}
Let $\displaystyle f(x)=\int_{0}^{x}L(t)t^{-1}dt$ and $U\in \mathcal{M}_{0}^{+}(L,\rho )$.\\[-4ex]
\begin{itemize}
\item[(i)] If $\rho >0$ and $f\in RV_{\alpha }$, $\alpha >0$, then 
$\displaystyle \log\int_{0}^{x}U(t)dt \underset{x\to\infty}{\thicksim}  \log U(x)$ and 
$\displaystyle \int_{0}^{x}U(t)dt\in \mathcal{M}_{0}^{+}(L,\rho )$.
\vspace{-3ex}
\item[(ii)] If $\rho <0$ and $f\in RV_{\alpha }$, $\alpha >0$, then $\displaystyle -\log
\int_{x}^{\infty }U(t)dt \underset{x\to\infty}{\thicksim} -\log U(x)$ and $\displaystyle \int_{x}^{\infty
}U(t)dt\in \mathcal{M}_{0}^{+}(L,\rho )$.
\end{itemize}
\end{proposition}
It seems to be quite difficult to consider integrals of functions in the class $\mathcal{M}_{0}^{+}(L,\rho )$ without extra conditions.

%
%

\section{The class $M_{1}(L,\protect\rho )$}
\label{sec:M1}

In this section, we are interested in studying a subclass of the large class $\mathcal{M}(\infty )$ of positive and measurable functions for which $\displaystyle \lim_{x\rightarrow \infty }\frac{\log U(x)}{\log x}=\infty$.  

We consider positive and measurable functions $U$ so that
\begin{equation} \label{9}
\lim_{x\rightarrow \infty }\frac{\log U(x)}{\int_{a}^{x}b^{-1}(t)dt}=\rho \in \R\text{,} 
\end{equation}
where $b(.)$ is a suitable function such that 
$\displaystyle b(x)/x \, \underset{x\to\infty}{\rightarrow} \,0$, to ensure that $U$ satisfying \eqref{9} belongs to the class $\mathcal{M}(\infty )$. 

Compared with \eqref{3}, we have $b^{-1}(x)=x^{-1}L(x)$ so that $L(x)\underset{x\to\infty}{\rightarrow} \infty $.

\begin{Example} Let us give examples of functions of $\mathcal{M}(\infty )$ satisfying \eqref{9}.\\[-4ex]
\begin{enumerate}
\item Let $U$ defined by $U(x)=\alpha e^{-\alpha x}$. Then  $\displaystyle \frac{\log U(x)}{x} \underset{x\to\infty}{\rightarrow} -\alpha $, and \eqref{9} holds with $b^{-1}(x)=1$.
\item Let $U$ be the standard normal density ($U(x)=c\,e^{-x^{2}/2}$). Then\\  
$\displaystyle \frac{\log U(x)}{x^{2}} \underset{x\to\infty}{\rightarrow} -\frac12$ and $b^{-1}(x)=x$ gives \eqref{9}.
\item Let $U(x)=e^{\left[ x\right] \log x}$, then $\displaystyle \frac{\log U(x)}{x\log x} \underset{x\to\infty}{\rightarrow} 1$ and  \eqref{9} holds with $b^{-1}(x)=\log x$.
\end{enumerate}
\end{Example}

In the next section we discuss the suitable functions $b(.)$ that we will
use.

\subsection{Definitions and Examples}

Let us recall a few definitions (see \cite{BinghamGT1989}, Sections 2.11 \& 3.10, \cite{BinghamO2014}, or \cite{Omey13}).
\begin{definition}
The positive and measurable function $b$ is called
'self-neglecting',  denoted by  $b\in SN$, if 
$$
b(x)/x\underset{x\to\infty}{\rightarrow} 0\quad\text{and}\quad
\lim_{x\rightarrow \infty }\frac{b(x+y\,b(x))}{b(x)}=1 \quad\text{locally uniformly in} \; y.
$$
\end{definition}
Note that 
$$
b\in SN \;\Rightarrow \; \lim_{x\to\infty}\int_{a}^{x}b^{-1}(t)dt =\infty 
\qquad \text{and}\qquad \lim_{x\to \infty} b^{\prime}(x)=0 \;\Rightarrow \; b\in SN.
$$

\begin{definition}~\\[-4ex]
\begin{itemize}
\item The positive and measurable function $f$ belongs to the class $\Gamma (b)$ if \\
$b\in SN$ and $f$ satisfies $\displaystyle \lim_{x\to\infty} \frac{f(x+yb(x))}{f(x)}=e^{y}$, for all $y$.
\item The positive and measurable function $f$ belongs to the class $\Gamma _{-}(b)$
if \\$1/f\in \Gamma (b)$.
\end{itemize}
\end{definition}

We can derive straightforward properties (see \cite{BinghamGT1989}, Section 3.10, and \cite{deHaan}).
\begin{Property}~\\[-4ex]
\begin{enumerate}
\item If $f\in \Gamma (b)$, then $\displaystyle \int_{a}^{x}f(t)dt\in \Gamma (b)$ and  $\displaystyle \int_{a}^{x}f(t)dt \underset{x\to\infty}{\thicksim} b(x)f(x)$. 
\item If $f\in \Gamma _{-}(b)$, then $\displaystyle \int_{x}^{\infty}f(t)dt\in \Gamma _{-}(b)$ and $\displaystyle\int_{x}^{\infty }f(t)dt \underset{x\to\infty}{\thicksim} b(x)f(x)$.
\item If $f\in \Gamma (b)$ then \eqref{9} holds locally uniformly in $y$ (the proof can be adapted from \cite{Omey13} and \cite{BinghamO2014}).
\end{enumerate}
\end{Property}

From Bingham et al. (see\cite{BinghamGT1989}, Theorem 3.10.8) or Omey (see \cite{Omey13}), we have the following representation theorem for functions in the class $\Gamma (b)$.
\begin{theorem}
$f\in \Gamma (b)$ for some function $b\in SN$ if and only if $f$ can be written as\;
$\displaystyle  f(x)=\exp \left\{ \alpha (x)+\int_{a}^{x}\frac{\beta (t)}{c(t)}dt\right\} $,
with\; $\alpha (x)\underset{x\to\infty}{\to} \alpha\in\R$, \\ $\beta (x)\underset{x\to\infty}{\to}  1
$ and where $c(.)$ is a positive, absolutely continuous function such that  $c^{\prime }(x) \underset{x\to\infty}{\to} 0$, and $c(x)\underset{x\to\infty}{\thicksim} b(x)$.
\end{theorem}

\begin{Example}~\\[-4ex]
\begin{enumerate}
\item The exponential density $f(x)=\alpha e^{-\alpha x}$ belongs to
the class $\Gamma _{-}(1/\alpha )$.
\item The standard normal density $f(x)=\frac1{\sqrt{2\pi}}\,e^{ -x^{2}/2}$ belongs to the
class $\Gamma _{-}(1/x)$.
\item $U(x)=e^{\left[ x\right] \log x}\in \mathcal{M}(-\infty )$ but $U\notin \Gamma $ (see Example 1.8 in \cite{CadenaT-2016}).
\end{enumerate}
\end{Example}

\begin{definition} Let $b\in SN$. The class $\mathcal{M}_{1}(b,\rho )$
is the set of positive and measurable functions $U$ such that \eqref{9} holds.
\end{definition}

Note that if $b\in SN$ and $U$ satisfies $b(x)U^{\prime }(x)/U(x)\underset{x\to\infty}{\rightarrow} \rho $, then $U\in \mathcal{M}_{1}(b,\rho )$.

\subsection{Some properties}

\subsubsection{General}

\begin{lemma}~\\
 For $\rho >0$, $\mathcal{M}_{1}(b,\rho )\subset \mathcal{M}(\infty )$, and for $\rho <0$,  $\mathcal{M}_{1}(b,\rho )\subset \mathcal{M}(-\infty )$.
\end{lemma}

\begin{proof}
We have
$$
\lim_{x\to\infty} \frac{\log U(x)}{\log x} =\lim_{x\to\infty}  \frac{\log U(x)}{\int_{a}^{x}b^{-1}(t)dt}\times \frac{\int_{a}^{x}b^{-1}(t)dt}{\log x} =\rho \times \lim_{x\to\infty}  \frac{\int_{a}^{x}b^{-1}(t)dt}{\log x}=
\rho \times \lim_{x\to\infty}\frac{x}{b(x)}. 
$$
Since $x/ b(x) \underset{x\to\infty}{\to}\infty$, the result follows. 
\end{proof}

In the next result we collect some algebraic results.
\begin{lemma}\label{lem:20}
Assume that $U\in \mathcal{M}_{1}(b_{1},\rho _{1})$.
\begin{itemize}
\item[(i)] If $V\in \mathcal{M}_{1}(b_{2},\rho _{2})$ and $b_{1}(x)/b_{2}(x) \underset{x\to\infty}{\to} c$, then $UV\in \mathcal{M}_{1}(b_{1},\rho _{1}+c\rho _{2})$ and 
$U/V\in \mathcal{M}_{1}(b_{1},\rho _{1}-c\rho _{2})$.
\item[(ii)] If $V\in RV_{\alpha }$, then $UV\in \mathcal{M}_{1}(b_{1},\rho _{1})$.
\item[(iii)] If $\rho _{1}\neq 0$ and $V\in \mathcal{M}_{1}(b_{2},\rho _{2})$
with $b_{2}$ bounded and $b_{1}(x)/b_{2}(x)\underset{x\to\infty}{\to} 0$, then
 $V(x)/U(x)\underset{x\to\infty}{\to} 0$ and $U+V\in \mathcal{M}_{1}(b_{1},\rho _{1})$.
\end{itemize}
\end{lemma}

\begin{proof}~
\begin{itemize}
\item[(i)] The result follows directly from the definitions.
\item[(ii)] We have
$$
\log U(x)V(x) =\log V(x)+\log U(x)=(\alpha +o(1))\log x+(\rho _{1}+o(1))\int_{a}^{x}b_{1}^{-1}(t)dt
$$
and $ \displaystyle \lim_{x\to\infty} \frac{\log x}{\int_{a}^{x}b_{1}^{-1}(t)dt}=\lim_{x\to\infty} \frac{b_{1}(x)}{x}=0$, hence the result.
\item[(iii)] We deduce from (i) that $U/V\in \mathcal{M}_{1}(b_{1},\rho _{1})$. It
follows that 
\[
\log \frac{U(x)}{V(x)}=(\rho _{1}+o(1))\int_{a}^{x}b_{1}^{-1}(t)dt,
\]%
which, combined with $\int_{a}^{x}b_{1}^{-1}(t)dt\rightarrow \infty $ (via the assumptions), gives that 
$V(x)/U(x)\underset{x\to\infty}{\to} 0$.\\
Now  we have, using (i),
$$
\log (U(x)+V(x))=\log U(x)+\log\left(1+V(x)/U(x)\right) \underset{x\to\infty}{\thicksim} \log U(x).
$$
Since $\displaystyle\int_{a}^{x}b_{1}^{-1}(t)dt\underset{x\to\infty}{\to}  \infty $, we obtain that \,
$\displaystyle
\frac{\log (U(x)+V(x))}{\int_{a}^{x}b_{1}^{-1}(t)dt} \underset{x\to\infty}{\to} \rho _{1}$.\\[-4ex]
\end{itemize}
\end{proof}

\subsubsection{Integrals}

Here we discuss integrals of functions in the class $\mathcal{M}_{1}(b,\rho )$. Note that we can alter the auxiliary function $b$ so that it satisfies $b^{\prime }(x)\underset{x\to\infty}{\to}  0$.

\begin{lemma}\label{lem:integral}
Assume that $U\in \mathcal{M}_{1}(b,\rho )$ with $b$ such that  $b^{\prime
}(x)\underset{x\to\infty}{\to}  0$.\\
If $\rho >0$, then $\displaystyle \int_{a}^{x}U(t)dt\in \mathcal{M}_{1}(b,\rho )$, and if $\rho <0$, then $\displaystyle \int_{x}^{\infty }U(t)dt\in \mathcal{M}_{1}(b,\rho) $.
\end{lemma}

\begin{proof}
Let assume that $\rho>0$. If $U\in \mathcal{M}_{1}(b,\rho )$, then for each $0<\epsilon <\rho $,
there exists $x_\epsilon$ such that $\forall x\ge x_\epsilon$,
\begin{equation}\label{10}
\exp\left\{(\rho -\epsilon)\int_{a}^{x} b^{-1}(t)\, dt\right\}\;\leq U(x)\leq \; 
\exp\left\{(\rho +\epsilon)\int_{a}^{x} b^{-1}(t)\, dt\right\},
\end{equation}
from which we deduce that, for $x\ge x_\epsilon$,
\[
\int_{x_\epsilon}^{x}L(t)\,dt \;\leq \int_{x_\epsilon}^{x}U(t)\,dt \;\leq \int_{x_\epsilon}^{x}R(t)\, dt 
\]%
where  $\displaystyle L(t)=\exp\left\{(\rho -\epsilon)\int_{a}^{x} b^{-1}(t)\, dt\right\}\in \mathcal{M}(b,\rho -\epsilon )$ and\\ $\displaystyle R(t)=\exp\left\{(\rho +\epsilon)\int_{a}^{x} b^{-1}(t)\, dt\right\}\in \mathcal{M}(b,\rho +\epsilon )$.
Clearly we have 
\[
\frac{R(t+y\,b(t))}{R(t)}=\exp\left\{(\rho +\epsilon)\int_{0}^{y}\frac{b(t)}{b(t+\theta\, b(t))}d\theta \right\} \underset{x\to\infty}{\to}  e^{(\rho +\epsilon) y} ,
\]
hence $R\in \Gamma (b/(\rho +\epsilon ))$, and $\displaystyle \int_{x_\epsilon}^{x}R(t)dt\underset{x\to\infty}{\thicksim} \frac{b(x)R(x)}{\rho +\epsilon }$.\\
Now, using the definition of $R$, we have 
\[
\log \frac{b(x)R(x)}{\rho +\epsilon }=\log b(x) -\log (\rho +\epsilon ) +(\rho +\epsilon
)\int_{a}^{x}b^{-1}(t)dt.
\]
Using l'Hopital's rule provides
$\displaystyle 
\lim_{x\rightarrow \infty }\frac{\log b(x)}{\int_{a}^{x}b^{-1}(t)dt}=\lim_{x\rightarrow \infty }b^{\prime }(x)=0$.\\
It follows that $\int_{x_\epsilon}^{x}R(t)dt\in \mathcal{M}_{1}(b,\rho +\epsilon )$.

Similarly, we obtain
$\displaystyle \int_{x_\epsilon}^{x}L(t)dt\underset{x\to\infty}{\thicksim}  \frac{b(x)L(x)}{\rho -\epsilon }$ and $\displaystyle \int_{x_\epsilon}^{x}L(t)dt\in \mathcal{M}_{1}(b,\rho -\epsilon )$. 

We can conclude, via \eqref{10}, that $\int_{x_\epsilon}^{x}U(t)dt\in \mathcal{M}_{1}(b,\rho )$.

The proof when $\rho<0$ follows the same steps.
\end{proof}

\begin{remark}~\\[-4ex]
\begin{itemize}
\item[(i)] In view of Lemma~\ref{lem:20}, (ii), the result given in Lemma~\ref{lem:integral} holds when replacing $U(x)$ by $x^{\alpha }U(x)$.
\item[(ii)] The proof of Lemma~\ref{lem:integral},(i), shows that for $\varepsilon >0$, there exists $x_\varepsilon$ such that, for all $x>x_\varepsilon$, 
\[
H^{\rho -\varepsilon }(x)\leq U(x)\leq H^{\rho +\varepsilon }(x)
\]
where $\displaystyle H(x)=\exp \left\{\int_{a}^{x}b^{-1}(t)dt\right\}\in \Gamma (b)$ satisfies 
$\displaystyle b(x)H^{\prime }(x)/H(x)=1$ (see Theorem~\ref{theo:16}).
\end{itemize}
\end{remark}

\begin{Example}~\\[-4ex]
\begin{enumerate}
\item Consider the tail distribution $\overline{F}(x)=e^{-x-0.5\sin x}$. \\
We have $\log \overline{F}(x)/x\underset{x\to\infty}{\rightarrow} -1$ and $\overline{F}\in \mathcal{M}_{1}(1,-1)$. 
The density, given by $f(x)=\overline{F}(x)(1+0.5\cos x)(>0)$, belongs to $\mathcal{M}_{1}(1,-1)$. We see that $b(x)f(x)/\overline{F}(x)$ does not converge to a limit as $x\to\infty$.

Note that $\overline{F}^{n}(a_{n}x)=\exp\{-na_{n}(x+0.5\sin(a_{n}x)/a_{n})\}$.
If $na_{n}\underset{n\to\infty}{\rightarrow}  1$, then $a_{n}\underset{n\to\infty}{\rightarrow} 0$ and $\overline{F}^{n}(a_{n}x)\underset{n\to\infty}{\rightarrow} e^{-1.5x}$.
\item Consider $U(x)=e^{x^{\beta }+\cos x}$, with $\beta >1$. Then $U^{\prime}(x)=U(x)(\beta x^{\beta -1}+\sin x)$ and $b(x)U^{\prime}(x)/U(x) \underset{x\to\infty}{\rightarrow}  1$.

Note that both $U$ and $U^{\prime }\in \mathcal{M}_{1}(b^{-1},1)$, with $b^{-1}(x)=\beta x^{\beta-1}$.
\end{enumerate}
\end{Example}

\subsubsection{Inverse functions}

\begin{proposition}
Let $U\in \mathcal{M}_1(b,\rho )$ and suppose that $U$ has a
derivative $U^{\prime }$ satisfying $\displaystyle b(x)U^{\prime }(x)/U(x) \underset{x\to\infty}{\rightarrow} \rho >0$.
Then the inverse function of U, denoted $V=U^{inv}$, belongs to $\mathcal{M}_{0}(L,1/\rho )$ with 
$L$ defined by $L(x)=b(V(x))/V(x)$.
\end{proposition}

\begin{proof}
We have
$\displaystyle  V^{\prime }(x)=\frac{1}{U^{\prime }(V(x))} $, so that
$\displaystyle V^{\prime }(x) \underset{x\to\infty}{\thicksim} \frac{b(V(x))}{\rho\,U(V(x)) }=\frac{b(V(x))}{\rho \,x}$.

It follows that, introducing  $b_0$ defined by $\displaystyle b_0(x)=xV(x)/b(V(x))$,
\[
\frac{b_0(x)b(V(x))}{xV(x)\rho } \underset{x\to\infty}{\thicksim} \frac{b_0(x)V^{\prime }(x)}{V(x)} 
\underset{x\to\infty}{\rightarrow} 1/\rho 
\]
and we see that, for $L(x)=b(V(x))/V(x)$,
$\displaystyle 
\frac{\log V(x)}{\int_{a}^{x}L(t)t^{-1}dt} \underset{x\to\infty}{\rightarrow} 1/\rho $.

Note that $L(U(x))=b(x)/x$, \, $L(x)\underset{x\to\infty}{\rightarrow} 0$ (since $b\in SN$), and \\
$\displaystyle \int_{a}^{x}L(t)t^{-1}dt=\int_{a}^{x}b(V(t))/(tV(t))dt$.
 It follows that
\[
\int_{a}^{x}L(t)t^{-1}dt=\int_{V(a)}^{V(x)}b(t)/(tU(t))U^{\prime
}(t)dt \underset{x\to\infty}{\thicksim} \rho \int_{V(a)}^{V(x)}1/t\, dt \underset{x\to\infty}{\to} \infty.
\]
Hence we obtain that $V\in \mathcal{M}_{0}(L,1/\rho )$.
\end{proof}

Note that we provided conditions to show that $U\in \mathcal{M}_{1}(b,\rho )$ implies that its inverse function $V\in \mathcal{M}_{0}(L,1/\rho )$. It is not clear if these additional assumptions can be
omitted.

\section{Concluding remarks}
The previous results easily extend to sequences. For a given sequence of
positive numbers $(b_{n})$, we can consider the class of sequences $(a_n)$ satisfying
for instance $\lim_{n\rightarrow \infty }\log
a_{n}/\sum_{k=1}^{n}k^{-1}b_{k}=\alpha $. If $(a_{n})$ is a regularly
varying sequence, we have that $\lim_{n\rightarrow \infty }\log a_{n}/\log
n=\alpha $, a constant.
\vspace{1.7ex}\\
We may also study $O-$type of results. Under Assumption~A, we may define 
$\rho _{L}$ and $\rho _{U}$ as follows:
\[
\rho _{L}=\liminf_{x\rightarrow \infty }\frac{\log U(x)}{\int_{a}^{x}L(t)t^{-1}dt}\;\text{and}\; \rho _{U}=\limsup_{x\rightarrow \infty } \frac{\log U(x)}{\int_{a}^{x}L(t)t^{-1}dt}\text{, with }\;\rho _{L}<\rho _{U}.
\]
This leads to inequalities of the form (see Corollary~\ref{cor:7}), with $V$ defined in \eqref{V},
\[
V^{\rho _{L}-\epsilon }(x)\leq U(x)\leq V^{\rho _{U}+\epsilon }(x).
\]

Finally note that many distribution functions $F$ and densities $f$ satisfy a
relation of the form%
\[
\lim_{x\rightarrow \infty }\frac{\log (1-F(x))}{\log f(x)}=1\text{.}
\]%
So it may be interesting to study functions $U$ satisfying the following relation:
\[
\lim_{x\rightarrow \infty }\frac{\log U(x)}{\log \left\vert U^{\prime
}(x)\right\vert }=1\text{.}
\]

\end{document}